\documentclass[10pt,a4paper]{amsart}
\usepackage{amscd,amssymb,amsopn,amsmath,amsthm,graphics,amsfonts,enumerate,verbatim,calc}
\usepackage[dvips]{graphicx}

\usepackage{amssymb,amsmath, color}

\headsep=1cm \numberwithin{equation}{section}
\newtheorem{theorem}{Theorem}[section]
\newtheorem{lemma}[theorem]{Lemma}

\newtheorem{proposition}[theorem]{Proposition}
\newtheorem{corollary}[theorem]{Corollary}

\theoremstyle{definition}

\theoremstyle{remark}
\newtheorem{remark}[theorem]{Remark}
\newtheorem{fact}[theorem]{Fact}
\newtheorem{example}[theorem]{Example}
\newtheorem{observation}[theorem]{Observation}
\newtheorem{discussion}[theorem]{Discussion}
\newtheorem{question}[theorem]{Question}
\newtheorem{conjecture}[theorem]{Conjecture}

\newtheorem{acknowledgement}{Acknowledgement}

\newcommand{\sat}{\operatorname{sat}}
\newcommand{\Se}{\operatorname{S}}
\newcommand{\HH}{\operatorname{H}}

\newcommand{\Ass}{\operatorname{Ass}}
\newcommand{\im}{\operatorname{im}}
\newcommand{\Ker}{\operatorname{ker}}

\newcommand{\grade}{\operatorname{grade}}

\newcommand{\Spec}{\operatorname{Spec}}

\newcommand{\rad}{\operatorname{rad}}

\newcommand{\Ht}{\operatorname{ht}}
\newcommand{\pd}{\operatorname{p.dim}}

\newcommand{\V}{\operatorname{V}}

\newcommand{\Ext}{\operatorname{Ext}}
\newcommand{\G}{\operatorname{G}}

\newcommand{\Sym}{\operatorname{Sym}}
\newcommand{\R}{\operatorname{R}}

\newcommand{\COH}{\operatorname{H}}

\newcommand{\depth}{\operatorname{depth}}

\newcommand{\lo}{\longrightarrow}
\newcommand{\fm}{\frak{m}}
\newcommand{\fp}{\frak{p}}
\newcommand{\fq}{\frak{q}}
\newcommand{\fa}{\frak{a}}

\usepackage{tikz}

\newcommand{\PP}{\mathbb{P}}

\begin{document}

\author[]{mohsen asgharzadeh }

\address{}
\email{mohsenasgharzadeh@gmail.com}

\title[ ]
{on the rigidity of symbolic powers }

\subjclass[2010]{ Primary  13D40;  13F20.}
\keywords{Symbolic powers; regular rings; local cohomology; rigidity.
 }

\begin{abstract}We deal with the rigidity
conjecture of symbolic powers over regular rings. This was asked by Huneke.
Along with our investigation, we  confirm a conjecture \cite[Conjecture 3.8]{9}.
\end{abstract}

\maketitle

\smallskip

\section{ Introduction}

In this note $(R,\fm)$ is a regular local ring
of dimension $d$.
Recall that the $n$-the \textit{symbolic power} of an ideal $I$ defined by $I^{(n)}:=\bigcap_{
\fp\in \Ass(I)}(I^nR_{\fp}\cap R)$. Recall from   \cite[Question 31]{hr} that:

\begin{question}
Let $\fp\in
\Spec(R)$. If $\fp^{(d)} = \fp^{d}$, does it follow that $\fp^n=\fp^{(n)}$
for all $n \geq 1$?
\end{question}

Question 1.1 is true in dimension $3$, see \cite[Corollary 2.5]{h}. This uses the \textit{intersection multiplicity} due to Serre. Also, Question 1.1 is true for $1$-dimensional Gorenstein prime ideals of a $4$-dimensional regular ring. This was proved by Huneke, see \cite[Corollary 2.6]{h}. Huneke and Ulrich extended this to the class of $1$-dimensional prime ideals of regular rings that are \textit{licci}, see \cite{hu}.
 In Theorem \ref{f} we show that:

\begin{theorem} Let $I$ be a Cohen-Macaulay height-two  ideal generated by exactly $d$ elements in  a regular local ring $R$ of dimension $d>2$.  Suppose $I$
is locally complete intersection. Then $I^i= I^{(i)}$  for all $i<d-1$ and $I^n\neq I^{(n)}\textit{ }\forall n\geq d-1.$
 \end{theorem}

In the graded situation and over polynomial rings, the claims $I^i= I^{(i)}$  for all $i<d-1$ and $I^{d-1}\neq I^{(d-1)}$ are in the  recent preprint \cite{9}.
One may try to drop the
conditions forced over $I$. Such a dream overlaps with the following recent conjecture  based on "\textit{computer experiments}":

\begin{conjecture}\label{conj} (See \cite{9})
 Let $X\subset \PP^3_k$
be a subvariety  of codimension $2$.
Assume that there is a point $\fp\in X$ such that the localization of $I_X$ at $\fp$ is not a complete
intersection. Then  $\forall m \geq 2$, the saturation of $I^m_X$ has an embedded
component at $\fp$. In particular, $I^m_X\neq I^{(m)}_X$  $\forall m \geq 2$.
\end{conjecture}

In Section 4 we show
Conjecture 1.3 is true in the  irreducible case.
We say an ideal $I$ is \textit{rigid}, if $I^n=I^{(n)}$ for all $n$ provided $I^{(i)} = I^{i}$ for all $i\leq \dim R$.
Recall from  \cite{hr}:
Is any prime ideal rigid?
Here, we present a sample:
\begin{corollary}
Any Cohen-Macaulay prime ideal   $\fp$ of height $2$ in a $4$-dimensional regular local ring  is rigid. In fact  $\fp^n=\fp^{(n)}$
for all $n \geq 1$ provided $\fp^{(3)} = \fp^{3}$.
\end{corollary}

In Section 5 we use the machinery of birational geometry  to  construct   non-rigid ideals (the ideals may  not be radical).

\section{Towards the rigidity in codimension one}

For the simplicity of the reader we collect some well-known results that we need:

\textbf{Subsection 2.A}: \textmd{Preliminaries.}
Recall that  an ideal $I$ is called \textit{complete intersection} if $I$ is generated
by a regular sequence of length equal to height of $I$.

\begin{lemma}(See \cite[Corollary 2.5]{h}) \label{par1}
Let $(R,\fm)$ be a regular local ring
of dimension $3$ and $\fp$ be a prime ideal of dimension one
which is not a complete intersection. Then $\fp^i\neq\fp^{(i)}$ for all $i>1$.
\end{lemma}

The grade of an ideal
$\frak a$ on a module $M$ is defined by $\grade_{A}(\frak a,M):=\inf\{i\in
\mathbb{N}_0|\Ext^i_R(R/{\frak a},M)\neq0\}.$ We use $\depth (M)$, when we deal with the maximal ideal of $\ast$-local rings.
Denote the minimal number of generators of $M$ by $\mu(M)$.

\begin{discussion}
An ideal $I\lhd R$ is called \textit{perfect} if $\pd(R/I)=\grade(I,R)$.
Also, $I$ is called \textit{strongly Cohen-Macaulay}, if its Koszul homologies $H_i(I,R)$ are all Cohen-Macaulay modules.
The ideal $I$ satisfies the $\G_\infty$ \textit{condition} if for all $\fp\in \V(I)$, one has $\mu (I_{\fp})\leq \Ht(\fp)$.
\end{discussion}

\begin{example}\label{av}(See \cite[Theorem 2.1(a)]{ah} and \cite[Supplement]{ah})\begin{enumerate}
\item[i)]
Perfect ideals of codimension two in a regular ring are strongly Cohen-Macaulay.
\item[ii)] If $\mu(I)\leq \Ht(I)+2$, $R$ is Gorenstein and $R/I$ is Cohen-Macaulay, then $I$ is strongly Cohen-Macaulay.
\end{enumerate}
\end{example}

\begin{lemma}\label{5}(See \cite[Lemma 2.7]{hu})
Let $R$ be a local Gorenstein ring, let $I$ be a perfect ideal
 which is strongly Cohen-Macaulay and $\G_\infty$. We write
$D := \dim(R/I)$. Then for all $n\geq d:=\mu (I)-\grade(I,R)$, we have
$\depth(R/ I^{n+1}) = D - d$.
\end{lemma}

An ideal $I$ is called  almost complete intersection, if $\mu(I)=\Ht(I)+1$. The definition presented in \cite{com} is more general than this.

\begin{lemma}\label{com}(See \cite[Theorem 3.1]{com})
Let $R$ be a Cohen-Macaulay ring and $\fp$
be a prime ideal which is an almost complete intersection. Then $\fp^{(2)}=\fp^2$ if and only if $\fp R_{\fq}$ is generated by an $R_{\fq}$-sequence
$\forall\fq\in\V(\fp)$
such that $\dim ( ( R / \fp ) _{\fq})\leq1$.
\end{lemma}

\textbf{Subsection 2.B}: \textmd{Towards the rigidity in codimension one.}
For an $R$-module $M$, the $i^{th}$ local cohomology of $M$ with
respect to an ideal $\fa$ is defined by
$\HH^{i}_{\fa}(M):={\varinjlim}_n\Ext^{i}_{R} (R/\fa^{n},
M)$.
 By definition,  $\HH^0_{\fm}(R/I^n)=\frac{\bigcup_j (I^n:\fm^j)}{I^n}=:\frac{(I^n)^{\sat}}{I^n}.$
 Note that $\fm\notin\Ass(R/I^n)$ if and only if $\depth(R/I^n)>0$ if and only if $I^n=(I^n)^{\sat}$.
In   fact, our interest in 1-dimensional  ideals
coming from:

\begin{fact}\label{tr}
Let $I\lhd R$  be a radical ideal of dimension one. Then $I^{(n)}/ I^n=\HH^0_{\fm}(R/I^n)$.
\end{fact}

In the case $I^{(n)}/ I^n=\HH^0_{\fm}(R/I^n)$, the equality of symbolic powers and ordinary powers translates to the positivity of the
depth function $f(n):=\depth(R/I^n)$.

\begin{proposition}\label{thm}
Let $(R,\fm)$ be a regular local ring
of dimension $d$ and $\fp$ be a prime ideal of dimension one
generated by at most $d$ elements. The following holds:
  \begin{enumerate}\item[a)]
Suppose $\fp^i=\fp^{(i)}$ for some $i>1$. Then
\\
\textmd{i)} $\fp^i=\fp^{(i)}$ for all $i$.\\
\textmd{ii)} $\fp$  generated by exactly $d-1$ elements.
\item[b)]
Suppose $\fp^{i_0}\neq\fp^{(i_0)}$ for some $i_0>1$. Then $\fp^i\neq\fp^{(i)}$ for all $i>i_0$.\end{enumerate}
\end{proposition}

\begin{proof}
 a): As $\fp$ is almost complete intersection and of dimension one, one can easily check that $\fp$
is perfect, strongly Cohen-Macaulay (see Example \ref{av}(ii)) and $\G_\infty$.
\begin{enumerate}
\item[i)]
In view of
Lemma \ref{5}, $f(n):=\depth(R/\fp^n)$ is constant for all $n>1$.
Having Fact \ref{tr}
in mind, $f(i_0)\neq0$. So, $f(n)\neq0$ for all $n>1$.
As $f(n)\leq \dim R/\fp^n=1$, we get $f(n)=1$ for all $n>1$. Again by Fact \ref{tr},
$\fp^i=\fp^{(i)}$ for all $i$.

\item[ii)]
By part $i)$, $\fp^2=\fp^{(2)}$.
By Lemma \ref{com},
$\fp_{\fq}$ is generated by an $R_{\fq}$-sequence for all
$\fq\in \V(\fp)$ such that $\dim R_{\fq}/\fp_{\fq}\leq 1$. For example we can take $\fq:=\fm$. The proof is now complete.
\end{enumerate}

b): Every $d$-generated prime ideal of height $d-1$  in a regular local ring of dimension $d$
is generated by a weak d-sequence, see \cite[1.3]{dseq}.  In view of \cite[Corollary 2.5]{dseq}, for  an ideal generated by a  weak d-sequence over an
integral domain, one has
$$\Ass(R/\fp)\subseteq\Ass(R/\fp^2)\subseteq\ldots\subseteq\Ass(R/\fp^{i_0})\subseteq\ldots\quad(\ast)$$
As $\fp^{i_0}\neq\fp^{(i_0)}$ and by Fact \ref{tr}, $\depth(R/\fp^{i_0})=0$. Thus  $\fm\in\Ass(R/\fp^{i_0})$. Applying $(\ast)$, we deduce that
$\fm\in\Ass(R/\fp^{i})$ for all $i>i_0$. In view of Fact \ref{tr},
$\fp^i\neq\fp^{(i)}$ for all $i>i_0$.
\end{proof}

The following
result shows that the assumption $\mu(\fp)=d$ is really important.

\begin{lemma}\label{par}(See \cite{hu})
Let $(R,\fm)$ be a regular local ring
of dimension $d\geq 4$ and $\fp$ be a prime ideal of dimension one
which is not a complete intersection and is in the linkage class of  a complete intersection.\begin{enumerate}
\item[i)]If  $R/\fp$ is  Gorenstein, then $\fp^2=\fp^{(2)}$ and $\fp^i\neq\fp^{(i)}$ for all $i>2$.
\item[ii)] If $R/\fp$ is not Gorenstein, then $\fp^i\neq\fp^{(i)}$ for all $i\geq2$.
\end{enumerate}
\end{lemma}

It was asked in \cite{e}  when is the natural map
$\Ext^*_R(R/\fa,R)\lo\HH^*_{\fa}(R)$ is  injective.
\begin{corollary}\label{12}
Let $\fp$ be as Lemma \ref{par}. Then, the natural map
$\Ext^d_R(R/\fp^n,R)\to \HH^{d}_{\fp} (R)$ is not injective  for all $n>2$.
\end{corollary}

\begin{proof}
 We are going to apply  Hartshorne-Lichtenbaum vanishing theorem. As regular local rings
are analytically irreducible and $\dim R/ \fp=1$ we get that $\HH^d_{\fp}(R)=0$. Thus, we need to show
$\Ext^d_R(R/\fp^n,R)\neq 0$.
Note that
$\Ext^d_R(R/\fp^n,R)\neq 0$ if its Matlis dual $\Ext^d_R(R/\fp^n,R)^v$
is nonzero. By \textit{local duality} and Fact \ref{tr}, we get that $\Ext^d_R(R/\fp^n,R)^v\simeq \HH^0_{\fm}(R/\fp^n)=\fp^{(n)}/\fp^n,$
which is nonzero.
\end{proof}

\section{Towards the rigidity in codimension two}
The  following well-known example (see e.g. \cite{tra}) illustrates our's idea:

\begin{example}
Let $R:=\mathbb{C}[x_0,\ldots,x_3]$ and let $\fp:=I_2(A)$ where \begin{equation*}
A:= \left(
\begin{array}{cccc}
x_0 & x_1 & x_2  \\
x_1 & x_2   & x_3 \\
\end{array} \right).
\end{equation*}
Then $\fp$ is  perfect of height two in a four-dimensional ring with three generators. Also, $\fp^n= \fp^{(n)}$ for all $n$.
\end{example}

An ideal  $I$ is called \textit{generically complete intersection} if $I_{\fp}$ is complete intersection for all $\fp\in \min(I)$.
Also, recall that $I$ is called \textit{locally complete intersection} if $I_{\fp}$ is complete intersection for all $\fp\in \V(I)\setminus\{\fm\}$.
In the graded situation we assume in addition that $\fp$ is homogeneous.  Our interest on locally complete intersection comes from the following
   non-linear version of \cite[Remark 2.2]{9}:

\begin{fact} \label{locc}
Let $I$ be a locally complete intersection ideal. Then $I^{(n)}/ I^n=\HH^0_{\fm}(R/I^n)$.
\end{fact}

\textmd{Sketch of proof:} Recall that symbolic powers and ordinary  powers are the same if the ideal is complete-intersection.
Also, localization behaves nicely with respect to symbolic power (see Lemma \ref{st} below).  It turns out that $I^{(n)}=(I^n)^{\sat}$.
It remains to mention  that $\HH^0_{\fm}(R/I^n)=\frac{(I^n)^{\sat}}{I^n}$.   \ \  \ \  $\Box$

\begin{lemma}\label{two}
Let $I$ be a Cohen-Macaulay height two  ideal generated by $d$ elements in  a regular local ring $R$ of dimension $d$.  If $I$
is locally complete intersection, then
$\depth (R/I^n)=\depth  (R/I^d)$ for all $\forall n\geq d-1$.
\end{lemma}

\begin{proof}
It is easy to see that such an ideal is perfect and $\G_\infty$. Due to Example \ref{av}, $I$ is strongly Cohen-Macaulay. It remains to apply Lemma \ref{5}.
\end{proof}

The same argument suggests:

\begin{corollary}\label{twoc}
Let $\fp$ be a Cohen-Macaulay height two prime ideal generated by $3$ elements in  a regular local ring $R$ of dimension $4$.  Then
$\depth (R/\fp^n)=\depth  (R/\fp^4)$ for all $n\geq 2$.
\end{corollary}

For each $X\subseteq \Spec(R)$ and any $\ell$, set $X^{\geq\ell}:=\{\fp\in X: \Ht(\fp)\geq \ell\}$.

\begin{lemma}\label{twoc1}
Let $I$ be a Cohen-Macaulay height two ideal locally complete intersection in  a regular local ring $R$ of dimension  $d=\dim R>2$.  Then $I^i=I^{(i)}$  for all $i>1$ if and only if $\mu(I)\leq d-1.$
\end{lemma}

\begin{proof}Suppose first that
$I^i=I^{(i)}$ for all $i>1$. The ideal $I$ is prefect, because $\pd(I)<\infty$. Clearly, a locally complete intersection ideal is generically complete intersection. In particular, we are in the situation of \cite[Corollary 3.4]{mor}. We note that
the proof of \cite[Corollary 3.4]{mor} uses the existence of a prime ideal of height $\geq 3$. Here, is the place that we use the assumption $d>2$. Now, \cite[Corollary 3.4]{mor} shows that $\mu(I)\leq d-1$.

Now, suppose $\mu(I)\leq d-1.$ Recall that a locally complete intersection ideal is generically complete intersection.  In view
 of Example \ref{av}, $I$ is strongly Cohen-Macaulay. Since $I$ is locally complete intersection, $\mu(IR_{\fq})=\grade(IR_{\fq},R_{\fq})=\Ht(IR_{\fq})=2\leq \Ht(\fq)-1$
for all $\fq\in \V(I)\setminus\{\fm\}$ with $3\leq\Ht(\fq)$. Also, $\mu(I)\leq d-1=\Ht(\fm)-1$. In particular, $\mu(I_{\fq})\leq \Ht(\fq)-1$  for all $\fq\in\V(I)^{\geq3}.$  By \cite[Lemma 3.1]{mor}, $I^i=I^{(i)}$ for all $i$.
\end{proof}

 The assumption $d>2$  is important:
Let $R:=k[[x,y]]$ and $I:=(x,y)$. Then  $I^n=I^{(n)}$ for all $n$. But, $\mu(I)>d-1$
(in the case of projective geometry there is no restriction on the dimension, because in the construction of $\PP^n$ we disregard the irrelevant ideal).

\begin{theorem}\label{f}
Let $I$ be a Cohen-Macaulay height two  ideal generated by exactly $d$ elements in  a regular local ring $R$ of dimension $d>2$.  Suppose $I$
is locally complete intersection. Then $I^i= I^{(i)}$  for all $i<d-1$ and $I^n\neq I^{(n)}\textit{ }\forall n\geq d-1.$
\end{theorem}

We barrow some lines from \cite{9}.

\begin{proof}
Since $I$ is locally complete intersection and in view of  Fact \ref{locc}, we observe that $I^{(n)}/ I^n=\COH^0_{\fm}(R/I^n)$.
This means that $I^i=I^{(i)}$ if and only if
$\depth(R/I^i)>0$. It is easy to see that $I$ is $\G_\infty$ and strongly Cohen-Macaulay. Combining these along with Lemma \ref{two}, we observe $$I^i= I^{(i)}\textit{ }\exists i\geq d-1\Longleftrightarrow
I^i= I^{(i)}\textit{ }\forall i\geq d-1. \quad(\ast)$$

Since $I$ is perfect, $\pd(R/I)=\grade(I,R)=\Ht(I)=2$. So, $\pd(I)=1$, and recall that $\mu(I_{\fp})\leq\Ht(\fp)=\depth(R_{\fp})$,
 because $I$ is locally complete intersection assumption.
These assumptions imply that
$\Sym^i(I)\simeq I^i$ for all $i\leq d$ (see \cite[Theorem 5.1]{ty}).
Due to Hilbert-Burch, $0\to\R^{d-1}\stackrel{(a_{ij})}\lo R^d\to I\to 0$ is exact. In the light of  \cite{wey}, the  following (not necessarily exact) complex $$F_i:0\lo \bigwedge ^i(R^{d-1})\lo \bigwedge ^{i-1}(R^{d-1})\otimes R^d\lo \ldots \lo R^{d-1}\otimes \Sym^{i-1}(R^{d})\lo\Sym^i(R^d)$$ approximates $\Sym^i(I)\simeq I^i$.

Let $C:=\Ker \left(\Sym(R^{d})\to \Sym(I)\right)=\im\left(\Sym(R^{d-1})\to \Sym(R^{d})\right)$.  Note that $\mathcal{R}:=\Sym(R^d)=R[X_1,\ldots,X_d]$. Thus, $C=(\sum_ia_{ij}X_i:1\leq j\leq d-1)$ when we view
 $C$ as an ideal of $\mathcal{R}$.
 Denote the Koszul complex with respect to a generating set of $C$ in $\mathcal{R}$ by $\textsc{K}$. We should remark that $\mathcal{R}$ is a graded polynomial ring.
The $i$-th spot of $\textsc{K}$ is the above complex $F_i$. Let $\fp\in \V(I)\setminus\{\fm\}$. To prove $(F_i)_{\fp}$ is exact, we need to show
$\textsc{K}_{\fp[X_1,\ldots,X_d]}$ is exact.
Such a thing is the case if $C_{\fp[X_1,\ldots,X_d]}$ is complete intersection. Let us  check this.
There are  regular elements $f,g$ such that $I_{\fp}=(f,g)$. Then $0\to R_{\fp}\to R_{\fp}^2$ approximates $I_\fp$. Denote
the identity matrix by $\textbf{Id}$. Thus
$$0\to\R_{\fp}^{d-1}\stackrel{(a_{ij})}\lo R_{\fp}^d\to I_{\fp}\to 0\simeq(0\to R_{\fp}\stackrel{\left(^{+f}_{-g}\right)}\lo R_{\fp}^2\to I_{\fp}\to 0)\oplus(0\to\R_{\fp}^{d-2}\stackrel{\textbf{Id}}\lo\R_{\fp}^{d-2}\to 0\to 0).$$
This yields that $C_{\fp[X_1,\ldots,X_d]}=(fX_1-gX_2,X_3\ldots,X_d)$ which is complete intersection. In sum,
 we observed that $F_i$ have finite length homologies.

 Suppose $i<d$. Then, in view of the \textit{new intersection theorem} \cite{int},
 $F_i$ is acyclic.  Conclude from Auslander-Buchsbaum formula that $\depth(R/I^i)>0$ for all $i<d-1$.
 This implies that $$I^i= I^{(i)}\textit{ }\forall i<d-1.$$ Suppose on the contradiction that
$I^n= I^{(n)}$ for some $n\geq d-1.$
Via the last displayed item and $(\ast)$, we deduce  that $I^n= I^{(n)}$ for all $n$.
Then by Lemma \ref{twoc1}, $\mu(I)\leq d-1$ which is a contradiction.
\end{proof}

\begin{corollary}\label{f1}
Let $I$ be a Cohen-Macaulay height two  ideal generated by at least $d$ elements in  a regular local ring $R$ of dimension $d>2$.  Suppose $I$
is locally complete intersection. Then $I^i= I^{(i)}$  for all $i<d-1$ and $I^{d-1}\neq I^{(d-1)}$.
\end{corollary}

\begin{proof}
Recall that $\mu(I_{\fp})\leq\Ht(\fp)=\depth(R_{\fp})$ for all $\fp\in\V(I)$ of height less than $d$ and that $\pd(I)=1$.
Therefore,
$\Sym^i(I)\simeq I^i$ for all $i\leq d-1$ (see \cite[Theorem 5.1]{ty}). This follows by above proof that $I^i= I^{(i)}$  for all $i<d-1$.
Suppose on the contradiction that $I^{d-1}= I^{(d-1)}$. Let us summarize things:
$I^i= I^{(i)}$  for all $i\leq d-1$, $I$ is perfect of height-two   and $I$
is generically  complete intersection. Under these assumptions   \cite[Theorem 3.2]{mor} implies that $I^i= I^{(i)}$  for all $i$.
In the light of Lemma \ref{twoc1} $\mu(I)\leq d-1.$ This is excluded by the assumption. This contradiction
implies that  $I^{d-1}\neq I^{(d-1)}$.
\end{proof}

The assumption $d>2$ is important:
Let $R:=k[[x,y]]$ and $I:=(x,y)$. Then $\mu(I)=2>d-1$.   Clearly, $I^2=I^{(2)}$.
The locally complete intersection assumption is important:

\begin{example}\label{loc}
 Look at $I:=(yzw, xzw, xyw ,xyz)$ as an ideal in $R:=k[x,y,z,w]$. Then
$I$ is a Cohen-Macaulay height two  ideal generated by exactly $4$ elements in  a $\ast$-local regular ring $R$ of dimension $4$.
It is easy to observe that $x^m yzw\in  I^{(2)}\setminus I^2$ for all large $m$. Thus
$I^{(2)}\neq I^{2}$.
\end{example}

The height-two assumption is important:

\begin{example}Set $R:=k[x_1,\ldots,x_5]$.
We look at the pentagon as a simplicial complex. Its Stanley-Reisner ring is  $R_\Delta:=k[\underline{x}]/I_\Delta:= R / (x_1x_3,x_1x_4,x_2x_4,x_2x_5,
x_3x_5) .$ Note that $I_\Delta$ is a height three ideal. So, $\dim(\Delta)=\dim R_\Delta-1=1$.
By the help of Macaulay 2, the projective resolution of $R_\Delta$ over $R$ is
$ 0\to R\to R^4  \to  R^5\to R \to R_\Delta\to 0$. Due to Auslander-Buchsbaum formula, $\depth_R(R_\Delta)=2=\dim (R_\Delta)$.
Thus $I_\Delta$ is perfect and of codimension $3$ generated minimally by $5$ elements. In view of  \cite[Proposition 1.11]{ter},
$I_\Delta$ is locally complete intersection. Thus,
for all $\fp\in \V(I)$, one has $\mu (I_{\fp})\leq \Ht(\fp)$. By definition, $I$ is $\G_\infty$.  Thanks to  Macaulay2,  $\pd(R/I_\Delta^3)=5$. By Auslander-Buchsbaum formula, $\depth(R/I_\Delta^3)=0$. In the light of Lemma \ref{5}, $\depth(R/I_\Delta^n)=0$ for all $n\geq 3$. Consequently, $I_\Delta^n\neq I_\Delta^{(n)}$ for all $n\geq 3$, because $I_\Delta$  is locally complete intersection (see Fact  \ref{locc}).
By \cite[Example 2.8]{ter},  $\pd(R/I_\Delta^2)=3$ and consequently $I_\Delta^2=I_\Delta^{(2)}$.
\end{example}

In fact, the above example suggests:

\begin{corollary}
Let $I$ be  a perfect ideal of height $d-2$ with minimally $d$ generators in a $d$-dimensional regular local ring $R$.
If $I$ is locally complete intersection, then $I^n\neq I^{(n)}$ for all $n\geq 3$.
\end{corollary}

\begin{proof}
As $\mu(I)-\grade(I,R)\leq2$, $I$ is strongly Cohen-Macaulay, because of Example \ref{av}. Since $\mu (I_{\fp})\leq \Ht(\fp)$, we see $I$ is $\G_\infty$. By Lemma \ref{5},
$\depth(R/ I^{n}) =0$  for all $n\geq 3$. Since $I$ is locally complete intersection and in view of  Fact \ref{locc}, we observe that
$I^{(n)}/ I^n=\HH^0_{\fm}(R/I^n)$. So,  $I^n\neq I^{(n)}$  for all $n\geq 3$.
\end{proof}

The Cohen-Macaulay assumption is important:

\begin{example}
Let $\mathcal{C}$ be the curve in $\PP^3$ parameterized by $\{s^4,s^3t,st^3,t^4\}$. This is the  \textit{Macaulay}'s curve. Denote the ideal
of definition  of $\mathcal{C}$  by $\fp$ which is a prime ideal in $ k[x,y,z,w]$. This is well-known that $\mu(\fp)=4$, $\fp^n=\fp^{(n)}$ for all $n$ and that
$R/\fp$ is not Cohen-Macaulay, see \cite[Example]{tra}. One can show that $\fp$ is locally complete intersection (for a quick proof please see Theorem \ref{tst}, below).
Also, $\Ht(\fp)=2$.
In particular, the Cohen-Macaulay assumption in Theorem \ref{f} is really important.
\end{example}

\section{Towards the rigidity in dimension four}

\begin{conjecture}\label{4.1}
(See \cite[Conjecture 3.8]{9}) Let $X\subset \PP^3_k$
be a subvariety (reduced and unmixed) of codimension $2$.
Assume that there is a point $\fp\in X$ such that the localization of $I_X$ at $\fp$ is not a complete
intersection. Then  $\forall m \geq 2$, the saturation of $I^m_X$ has an embedded
component at $\fp$. In particular, $I_X^m\neq I_X^{(m)}$ for all $m\geq 2$.
\end{conjecture}

\begin{observation}
The monomial-situation rarely happens:
Recall that  $I$ is an ideal in the ring $R:=k[x_1,\ldots,x_4]$.
As $I$ is radical  unmixed  monomial and 2-dimensional  we have $$I=\rad(I)=\bigcap_{\textit{for some pairs } (i\neq j)} (x_i,x_j).$$
Let $G$ be the \textit{graph} with the vertex set $\{1,\ldots,4\}$ where $\{i,j\}$ is an edge, if
such a pair does not appear in the above intersection. Suppose on the contradiction that $I^m= I^{(m)}$ for some $m\geq 2$.
Thanks to \cite[Lemma 3.1]{edg}, $G$ is a path
or a cycle or the union of two disjoint edges.
In the case of union of two disjoint edges, $I$
is locally complete intersection (see \cite[Example 1.18]{ter}) which is excluded by the conjecture.  The corresponding ideal of the paw-graph

$$
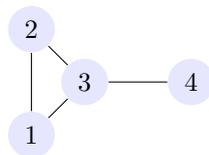
\begin{tikzpicture}
  [scale=.7,auto=left,every node/.style={circle, fill=blue!10}]
  \node (n1) at (1,5) {1};
  \node (n2) at (1,7)  {2};
  \node (n3) at (2,6)  {3};
  \node (n4) at (4,6)  {4};
 \foreach \from/\to in {n1/n2,n1/n3,n2/n3,n3/n4}
    \draw (\from) -- (\to);
\end{tikzpicture}$$
is $(x_1x_4,x_2x_4)$ which is of height one. This is excluded.  The  ideal of the diamond

$$
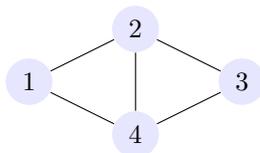
\begin{tikzpicture}
  [scale=.7,auto=left,every node/.style={circle,fill=blue!10}]
  \node (n1) at (1,2) {1};
  \node (n2) at (3,3)  {2};
  \node (n3) at (5,2)  {3};
  \node (n4) at (3,1)  {4};
 \foreach \from/\to in {n1/n2,n1/n4,n2/n3,n3/n4,n2/n4}
    \draw (\from) -- (\to);
\end{tikzpicture}$$
is $(x_1x_3)$ which is of height one. This is excluded.  The corresponding ideal of

$$
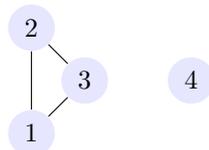
\begin{tikzpicture}
  [scale=.7,auto=left,every node/.style={ circle,fill=blue!10}]
  \node (n1) at (1,5) {1};
  \node (n2) at (1,7)  {2};
  \node (n3) at (2,6)  {3};
  \node (n4) at (4,6)  {4};
 \foreach \from/\to in {n1/n2,n1/n3,n2/n3}
    \draw (\from) -- (\to);
\end{tikzpicture}$$
is $(x_1x_4,x_2x_4,x_3x_4)$ which is of height one. This is excluded.  Also, the corresponding ideal of tetrahedral graph
excluded. Therefore,  $G$ is either 4-gon (square graph) or $4$-pointed path. In particular,
it is connected. Deduce from \cite[Proposition 1.11]{ter} that $I$ is locally complete intersection which is excluded by the conjecture.
\end{observation}

The following result is well-known (see e.g. the stacks project).

\begin{lemma}\label{st}
Let $R\to S$ be a flat ring map (e.g.  localization with respect to a multiplicative closed set). Let $\fq\lhd R$ be a prime ideal such that $\fp=\fq S$ is a prime ideal of $S$. Then $\fq^{(n)}S=\fp^{(n)}$.
\end{lemma}

 \begin{theorem}\label{tst}
Conjecture \ref{4.1} is true for irreducible varieties.
\end{theorem}

\begin{proof}
Let $\fq$ be the defining ideal of the variety $X\subset\PP^3$. Then, $\fq$ is a  height-two prime ideal in $R:=k[x_1,\ldots,x_4]$
which is not locally complete intersection. By definition, there is a homogeneous  prime ideal $\fp\in\V(\fq)\setminus\{\fm\}$
such that $\fq R_{\fp}$ can not be generated by a regular sequence.
Note that $\fq R_{\fq}$ is generated by a regular sequence, because it is a maximal ideal of a regular local ring.
Deduce from this that $\fp\supsetneqq \fq$. One may find easily that $\fp\subsetneqq \fm$. We conclude by this that $\Ht(\fp)=3$.
We summarize things as follows: $R_{\fp}$ is a regular local
ring of dimension $3$ and $\fq R_{\fp}$  is a prime ideal of height two which is not a complete intersection ideal. Set $\mathcal{S}:=R\setminus \fp$. In the light of Lemma \ref{par1}, $$(\mathcal{S}^{-1}\fq)^{m}\neq(\mathcal{S}^{-1}\fq)^{(m)}\quad \forall m>1\quad(\ast)$$One can find easily that $(\mathcal{S}^{-1}\fq)^{m} =\mathcal{S}^{-1}(\fq^{m})$  . If $\fq^{m}=\fq^{(m)}$ were be the case, in view of Lemma \ref{st}, we should have $$(\mathcal{S}^{-1}\fq)^{m} =\mathcal{S}^{-1}\fq^{m}=\mathcal{S}^{-1}\fq^{(m)}=(\mathcal{S}^{-1}\fq)^{(m)},$$ which is a contradiction via $(\ast)$. So, $\fq^{m}\neq\fq^{(m)}$ for all $m>1$.
\end{proof}

\begin{corollary}\label{four}
Let $\fp$ be a $2$-dimensional Cohen-Macaulay prime ideal in a  regular local ring  of dimension four. If $\fp^{(3)} = \fp^{3}$, then $\fp^n=\fp^{(n)}$
for all $n \geq 1$. 
\end{corollary}

\begin{proof}
By the proof of Theorem \ref{tst} we may assume that $\fp$  is locally complete intersection. If $\mu(\fp)\leq 3$, by the help
of Lemma \ref{twoc1},  we observe that $\fp^n=\fp^{(n)}$
for all $n \geq 1$. Suppose $\mu(\fp)\geq4$. Then we are in the situation of Corollary \ref{f1}. In view of Corollary \ref{f1}, $\fp^{(3)} \neq \fp^{3}$.
This is excluded by the assumptions. The proof is now complete.
\end{proof}

\section{Towards non-rigidity}

The rings in this section are of zero characteristic.

\begin{discussion}\label{disb}
Let $X \subset \PP^n$ be a
projective variety. Given a rational map $\textsc{F} :X\dashrightarrow \PP^n$ into another projective space,
$Y \subset \PP^n$ denote its image. Recall that $\textsc{F}$ is called birational onto its image if there exists a
rational map $\textsc{G} : Y\dashrightarrow \PP^n$ whose image is $X$. When such a thing happens we say $\textsc{F}$ is a \textit{Cremona transformation}.
Note that  $\textsc{F}$ (resp. $\textsc{G}$) determines by forms $\underline{f}:=f_0,\ldots,f_n$ (resp. $\underline{g}$) of same degree. By $d$ (resp. $d'$)
we mean $\deg(f_i)$ (resp. $\deg(g_j)$).
Also, $\underline{g}$ is called \textit{representatives of the inverse}.
The ideal generated by $\underline{f}$ is called the \textit{base ideal}.
By a result of Gabber, $d'\leq d^{n-1}$, see \cite{Gabber}.
In particular, when $n=2$ one can recover the classical result  $d=d'$.
\end{discussion}

The following is a method to construct non-rigid ideals.

\begin{proposition}\label{sr}
Adopt the above notation and let $I$ be
the base ideal generated by forms of degree $d\geq2$. Assume the following conditions hold:
\begin{enumerate}
\item[i)]$\depth(R/I)>0 $,
\item[ii)]$I^{(\ell)}/I^{\ell} $ is either zero or $\fm$-primary for all $\ell$, and
\item[iii)] The Rees algebra $\R(I)$ satisfies Serre's condition $\Se_2$.
\end{enumerate}
Then
$I^{\ell}=I^{(\ell)}$ for all $\ell<d'$ and $I^{d'}\neq I^{(d')}$, where $d'$ is the degree of representatives of the inverse.
\end{proposition}

\begin{proof}
In the light of \cite[Theorem 2.1]{cos} we observe that
the symbolic Rees algebra $\R^{(I)}:=\bigoplus_s I^{(s)}t^s$ is equal to $R[It, Dt^{d'}]$,
where $D$ is called   the \textit{source inversion}.  We note that $D$ is defined by the equation $$g_i(f_0,\ldots,f_n)=Dx_i\quad \forall i\quad(\ast)$$In particular, for all $i<d'$, one has $$I^{(i)}=\R^{(I)}_i=\R(I)_i=I^i$$
By $(\ast)$, we have $\deg(D)=dd'-1$. Recall that $I^{d'}$ has no element of degree less than $dd'$. Thus, $D\in \R^{(I)}_{d'}\setminus I^{d'}=I^{(d')}\setminus I^{d'}$.
\end{proof}

We give a non-rigid ideal:

\begin{example}
The primeness of the ideal is important. Let $R:=k[x,y,z]$ and let $d$ be any integer. Take $I:=(x^d,x^{d-1}y,y^{d-1}z)$.
This is the base ideal of  $\PP^2\dashrightarrow \PP^2$ and is of degree $d$. The base ideal is one-dimensional and  saturated. In particular, it is Cohen-Macaulay. Under this assumption
it is shown in \cite[Proposition 3.4]{cos} that the conclusion of Proposition \ref{sr} holds true.
Recall from Discussion \ref{disb} that $d=d'$.
It follows  that $I^i=I^{(i)}$ for all $i<d$ and $I^d\neq I^{(d)}$.
\end{example}

\begin{discussion}\label{dis}
Let us revisit   $I:=(yzw, xzw, xyw ,xyz)$ as an ideal in $A:=k[x,y,z,w]$.
Then $I$ is  square-free and is $3$-Veronese. Such an ideal is perfect, of height 2, 4-generated
and the depth of its powers are computed by the following table (see \cite[Corollary 10.3.7]{hh})\begin{equation*}
\depth(R/I^n)= \left\{
\begin{array}{rl}
2 & \  \   \   \   \   \ \  \   \   \   \   \ \text{if } n=1\\
1 &\  \   \   \   \   \ \  \   \   \   \   \ \text{if } n = 2\\
0 & \  \   \   \   \   \ \  \   \   \   \   \ \text{if } n > 2
\end{array} \right.
\end{equation*}
\end{discussion}

Let us show that the assumption ii) in Proposition \ref{sr} is really needed.

\begin{example}
Look at $I:=(yzw, xzw, xyw ,xyz)$ as an ideal in $R:=k[x,y,z,w]$. We note that $I$ is the base ideal
of the Cremona map $\textsc{F}:\PP^3\dashrightarrow \PP^3$.

i) In view of Example \ref{loc}, the module $I^{(2)}/I^{2} $  is not zero.
By Discussion \ref{dis}, $\depth (R/I^2)=1$. Thus, $\fm\notin\Ass(R/I^2)$.
Therefore, $I^2$ has no $\fm$-primary component. So, $I^{(2)}/I^{2} $ is not $\fm$-primary.

ii) Here, we show that $I$ respects the conditions i) and iii) in  Proposition  \ref{sr}.
 In the light of
 \cite{vi},
 $\overline{I^{n}}=I^{n}$ for all $n$. This means that $\R(I)$ is normal. Normal rings are $\Se_2$. Also, $\depth(R/I)=2>0 $.

iii) Suppose on the contradiction that  Proposition  \ref{sr} is true without its second assumption.
Recall that $\textsc{F}$ defines via the partial derivations of $f:=xyzw$.
This is well-known and classical from birational geometry that $\textsc{F}$ is
\textit{standard involution}, i.e., it is self-inverse as a rational map. In particular, $d=d'=3$.
In view of Proposition  \ref{sr}
 we should have $I^{2}=I^{(2)}$ which is a contradiction with  Example \ref{loc}.
\end{example}

\begin{example}Let us apply things in an  example due to Dolgachev:
Let $f:=x(xz+y^2)$. In view of \cite[Page 192]{dol}, the partial derivations  $\{\partial f/ \partial x, \partial f/ \partial y, \partial f/ \partial z\}$ define  a plane Cremona transformation. Such a transformation is called\textit{ polar }transformation.
It is easy to find that the base ideal $I=(2xz+y^2,xy,x^2)$ is perfect and of codimension $1$. Recall that $d'=d=2$. By Proposition \ref{sr},
$I^2\neq I^{(2)}$.
\end{example}

\begin{remark} Let us look at the system
$\{X^2,XY,WX+Y^2,ZX+W^2,UX+Z^2\}$. This defines a birational map $\PP^4\dashrightarrow\PP^4$ such that $d'= d^{n-1}$. Let $A:=k[X,Y,W,Z,U]$ and let $I:=(X^2,XY,WX+Y^2,ZX+W^2,UX+Z^2)$. Set $R:=\frac{A}{I}$. Here, we only check that $\depth(R)>0$:
We use lowercase letters to represent the elements in $R$.
We show $u$ is a regular element.  It turns out that $u$ is a homogeneous parameter element.
Look at $R$ as a $k[u]$-module. The set $\Gamma:=\{1,x,y,w,z,y^2,yw,w^2,wz,w^3,ywz\}$ is a generating set for $R$ as a $k[u]$-module.
Since $\Gamma$ is linearly independent over $k[u]$, we observe that
$R$ is free as a $k[u]$-module. Clearly, $u$ is regular over $k[u]$.
So, $u$ is regular, as claimed.
\end{remark}

\begin{conjecture}\label{gcon}
Assume that $R$ is a $d$-dimensional polynomial  ring over a field  and  $I\lhd$ be an  ideal.
Suppose there is a polynomial function $f$ of degree at most $d$ with coefficients depend  only on the degree of generators of $I$  and $d$ such that
$I^{(i)} = I^{i}$ for all $i<f$. Then $I^n=I^{(n)}$
for all $n \geq 1$.
\end{conjecture}

\begin{acknowledgement}
I thank Prof. Simis for his  comments on the very earlier version of this note.
\end{acknowledgement}


\begin{thebibliography}{99}




\bibitem{ah}L. Avramov, and J. Herzog,
\emph{The Koszul algebra of a codimension 2 embedding},
Math. Z.  {\bf175} (1980),  249–-260.

\bibitem{Gabber}
H. Bass, E. Connell, and D. Wright, \emph{The Jacobian conjecture: reduction of degree and formal expansion
of the inverse}, Bull. Amer. Math. Soc. 7 (1982),  287-–330.


\bibitem{cos}
B. Costa, A. Simis,  and Z. Ramos, \emph{A theorem about Cremona maps and symbolic Rees algebras}, Internat. J. Algebra Comput. {\bf{24}} (2014), 1191–-1212.

\bibitem{9}S. Cooper, G. Fatabbi, E. Guardo, A. Lorenzini, J. Migliore, U. Nagel, A. Seceleanu, J. Szpond and A. Van Tuyl,  \emph{
Symbolic powers of codimension two Cohen-Macaulay
ideals}, arXiv:1606.00935[Math.AC]

\bibitem{dol}
I.V. Dolgachev,
\emph{Polar Cremona transformations},
Michigan Math. J. {\bf 48} (2000), 191--202.


\bibitem{e}D. Eisenbud, M. Mustata, and M. Stillman,
\emph{Cohomology on toric varieties and local cohomology with
monomial supports}, J. Symbolic Comput. {\bf 29} (2000), 583-–600.

\bibitem{mac}
D. Grayson,  and M. Stillman,
\emph{Macaulay2, a software system for research in algebraic geometry},
Available at http://www.math.uiuc.edu/Macaulay2/.



\bibitem{hh}
J. Herzog, and T. Hibi,  \emph{Monomial ideals, Graduate Texts in Mathematics}, {\bf260} Springer-Verlag London, Ltd., London, (2011)


 \bibitem{hu}
C. Huneke, and B. Ulrich, \emph{Powers of licci ideals}, Commutative algebra (Berkeley, CA, 1987), 339–-346, Math. Sci. Res. Inst. Publ., {\bf{15}}, Springer, New York, 1989.

\bibitem{h}C. Huneke, \emph{
The primary components and integral closures of ideals in 3-dimensional regular local rings}, Math. Ann.  {\bf{275}} (1986), 617–-635.

\bibitem{dseq}
C. Huneke, \emph{Powers of ideals generated by weak d-sequences}, J. Algebra {\bf{68}} (1981),  471–-509.

\bibitem{hr}
C. Huneke, and C. Raicu,
\emph{Introduction to uniformity in commutative algebra}, lectures given by the first author as part of an
introductory workshop at MSRI for the program in Commutative Algebra, 2012--13,  arXiv:1408.7098 [math.AC].

\bibitem{com}
C. Huneke, \emph{Symbolic powers of prime ideals and special graded algebras}, Comm. Algebra {\bf{9}} (1981),  339–-366.



\bibitem{edg}
N.C. Minh, and N.V. Trung, \emph{Cohen-Macaulayness of powers of two-dimensional squarefree monomial ideals}, J. Algebra {\bf{322}} (2009), 4219–-4227.

\bibitem{mor}
S. Morey, \emph{
Stability of associated primes and equality of ordinary and symbolic powers of ideals},
Comm. Algebra {\bf{27}} (1999), 3221–-3231.

\bibitem{stk}
Stacks project authors, \emph{Stacks project}, http://stacks.math.columbia.edu. (2016).


\bibitem{int}
P. Roberts, \emph{Le theoreme d' intersection}, C. R. Acad. Sci. Paris Ser. I Math.,
{\bf{304}}  (1987),   177–-180.


\bibitem{ter}
N. Terai, and K.I. Yoshida, \emph{
Locally complete intersection Stanley-Reisner ideals},
Illinois J. Math. {\bf{53}} (2009), 413–-429.


\bibitem{ty}
 A. Tchernev, \emph{Torsion freeness of symmetric powers of ideals}, Trans. Amer. Math. Soc.  {\bf{359}} (2007), 3357–-3367.

\bibitem{tra}N.V. Trung,
\emph{On a certain transitivity of the graded ring associated with an ideal},
Proc. Amer. Math. Soc. {\bf{85}} (1982), no. 4, 489–-495.

\bibitem{vi}
R.H. Villarreal, \emph{Normality of subrings generated by square free monomials}, J. Pure Appl. Algebra {\bf{113}} (1996),  91-–106.

\bibitem{wey}
J. Weyman, \emph{Resolutions of the exterior and symmetric powers of a module},
J. Algebra  {\bf{58}}, (1979) 333--341.

\end{thebibliography}
\end{document}